\documentclass[11pt]{article}
\usepackage{amsmath,amsthm,amssymb,amsfonts}
\usepackage{latexsym}
\usepackage{graphicx,psfrag,import}
\usepackage{fullpage}
\usepackage{framed}
\usepackage{verbatim}
\usepackage{color}
\usepackage{epsfig}
\usepackage{epstopdf}
\usepackage{hyperref}
\usepackage{geometry}
\usepackage{mathtools}
\usepackage{enumerate}
\usepackage{multicol}
\usepackage{a4wide}
\usepackage{booktabs}
\usepackage{enumitem}
\usepackage{lineno}
\usepackage{parcolumns}
\usepackage{thmtools}
\usepackage{xr}
\usepackage{epstopdf}
\usepackage{mathrsfs}
\usepackage{subfig}
\usepackage{caption}
\usepackage{comment}
\usepackage{authblk}
\usepackage{setspace}

\theoremstyle{plain}
\newtheorem{theorem}{Theorem}[section]

\newtheorem{proposition}[theorem]{Proposition}
\newtheorem{lemma}[theorem]{Lemma}

\theoremstyle{definition}
\newtheorem{definition}[theorem]{Definition}
\newtheorem{example}[theorem]{Example}

\newtheorem{remark}[theorem]{Remark}

\theoremstyle{remark}

\newcommand\RR{\mathbb{R}}
\newcommand\GG{\mathcal{G}}

\newcommand\V{\mathcal{V}}
\newcommand\R{\mathbb{R}}
\newcommand\by{\boldsymbol{y}}
\newcommand\bk{\boldsymbol{k}}
\newcommand\bw{\boldsymbol{w}}

\newcommand\bu{\boldsymbol{u}}
\newcommand\bx{\boldsymbol{x}}

\newcommand{\defi}{\textbf}
\DeclareMathOperator{\spn}{span}

\DeclareMathOperator{\cl}{cl}
\newcommand{\mS}{\mathcal{S}}

\begin{document}

\title{Extinction in Reaction Network Models}
\author[1]{Pranav Agarwal}
\author[2]{Gheorghe Craciun}
\author[3]{Abhishek Deshpande}
\author[4]{Jiaxin Jin}
\affil[1,3]{\small Center for Computational Natural Sciences and Bioinformatics, \protect \\ International Institute of Information Technology Hyderabad} 
\affil[2]{\small Department of Mathematics and Department of Biomolecular Chemistry, \protect \\ University of Wisconsin-Madison}
\affil[4]{\small Department of Mathematics, University of Louisiana at Lafayette}

\maketitle

\begin{abstract}
In this paper, we study extinction in dynamical systems generated by reaction networks. We introduce two notions: \emph{weak extinction} and \emph{strong extinction}, and relate them to the structure of the underlying network through Lyapunov functions and LaSalle’s invariance principle. 
In particular, for all deficiency-zero networks that are {\em not} weakly reversible, we provide a geometric construction of linear Lyapunov functions. Using these functions, we establish that if these networks have bounded invariant subspaces, then they must exhibit {\em weak} extinction within every such subspace. Also, for linear networks that are not weakly reversible, we show that every species outside a terminal strongly connected component undergoes {\em strong} extinction. Moreover, in order to further emphasize the difference between weak and strong extinction, we construct an example of a reaction system (based on the Ivanova network) that exhibits weak extinction for all the species, but does not exhibit strong extinction in any species. 
\end{abstract}

% \tableofcontents 

\section{Introduction}

Systems of polynomial equations frequently arise in models of biological interaction networks. A fundamental question in this context is whether a particular species can \emph{persist} or go \emph{extinct}. 
Persistence, defined as the property that no species goes extinct, has been studied extensively. 
Starting from the work of Feinberg \cite{feinberg1977chemical}, who conjectured that weakly reversible reaction networks are persistent, subsequent studies have established persistence in specific network classes. In particular, two-dimensional endotactic networks and weakly reversible networks with two-dimensional stoichiometric subspaces have been shown to be persistent \cite{craciun2013persistence, pantea2012persistence}, and strongly endotactic networks were proven to be persistent in \cite{gopalkrishnan2014geometric}.

Our focus here is on \emph{extinction}, which has biological significance. In epidemiology, it can be used to identify conditions for eradicating a certain type of infection, which is important for therapeutic design~\cite{dowdle1998principles, henderson2013lessons}. In the pharmaceutical industry, it can be used to selectively flush out a particular toxic species~\cite{bidlack2002casarett,goodman1996goodman,rowland2011clinical}. This highlights that extinction is not merely a theoretical outcome of a dynamical system, but also a clinically meaningful endpoint.
Determining when and how extinction occurs provides key insights into the long-term behavior of biological systems.

%For example, in viral infection research, identifying conditions under which a virus population is eradicated within a host is crucial for therapeutic design. As noted by M%bogo~\cite{mbogo2013hiv}, extinction analysis enables assessment of treatment effectiveness, since successful intervention must drive viral populations toward elimination rather than mere persistence. 

The study of extinction is closely linked to the dynamics and stability of these networks. Papers by May~\cite{may1972stable} showed that stability depends on the structure of interactions among species rather than on system complexity.  
Despite its relevance, the notion of extinction in reaction networks remains relatively underdeveloped. One line of work~\cite{johnston2018conditions} introduces an \emph{extinction event} for networks with a discrete state space, such as those modeled by discrete-state Markov chains. The approach constructs a modified network and verifies structural properties of this auxiliary network to determine whether the original system exhibits an extinction event. A computational package implementing these criteria was later introduced in~\cite{johnston2017computational}.

In this paper, we introduce the notions of \emph{weak extinction} and \emph{strong extinction}. We say that a mass-action system exhibits \emph{weak extinction} if for some initial condition, there exists a species $X_i$ such that the corresponding solution $\bx(t)$ satisfies
\[
\liminf\limits_{t \to \infty} x_i (t) = 0.
\]
Further, a mass-action system exhibits \emph{strong extinction} if for some initial condition, there exists a species $X_i$ such that the corresponding solution $\bx(t)$ satisfies
\[
\lim\limits_{t \to \infty} x_i (t) = 0.
\]
A detailed formulation of these two notions is provided in Definition~\ref{defn:strong_extinction}.

By definition, strong extinction in species $X_i$ implies weak extinction in species $X_i$. However, the converse does not hold. In Example~\ref{ex:modified_ivanova}, for a system with initial condition $\bx_0 \in \mathbb{R}^n_{>0}$, the corresponding solution $\bx(t)$ satisfies
\[
\liminf_{t \to \infty} x_i(t) = 0, 
\ \text{ but } \ \lim_{t \to \infty} x_i(t) \neq 0,
\]
illustrating that weak extinction can occur without strong extinction.

Here, we connect extinction with the theory of \emph{Lyapunov functions} and \emph{LaSalle’s invariance principle} (see Definition \ref{def:lyapunov_function} and Theorem \ref{thm:LaSalles_invariance}). 
In Section~\ref{sec:non_wr_lyapunov}, Theorem~\ref{thm:stronger_not_consistent} establishes four equivalent statements that relate the structure of the network, the dynamical properties of its mass-action systems, and the existence of a linear Lyapunov function.
This further implies the following equivalence (see details in Remark \ref{rmk:lyapunov_function}):

\begin{enumerate}[label=(\alph*)]
\item For all choices of rate constants $\bk$, there exists a linear Lyapunov function $\V$ for the mass-action system $(G, \bk)$ such that $\dot{\V}(\bx(t)) < 0$ for every $\bx \in \mathbb{R}^n_{>0}$.

\item For all choices of rate constants $\bk$, there exists a Lyapunov function $\V$ for the mass-action system $(G, \bk)$ such that $\dot{\V}(\bx(t)) < 0$ for every $\bx \in \mathbb{R}^n_{>0}$.
\end{enumerate} 

\noindent
This equivalence is somewhat unexpected, as it shows that one may restrict attention to \emph{linear} Lyapunov functions, significantly simplifying the analysis. We then provide an explicit construction of such a Lyapunov function for every deficiency-zero reaction network that is not weakly reversible (see Definitions \ref{def:deficiency} and \ref{def:graphic_property}) in the following proposition.

\begin{proposition}
\label{prop:def_0_non_wr}

Let $G = (V, E)$ be a deficiency-zero reaction network that is not weakly reversible. 
For any choice of rate constants $\bk$, the mass-action system $(G, \bk)$ admits a linear Lyapunov function $\V$ such that 
\[
\dot{\V}(\bx(t)) < 0 
\ \text{ for all } \ 
\bx \in \mathbb{R}^n_{>0}.
\]
\end{proposition}

In Section~\ref{sec:extinction_two_class}, we apply the Lyapunov function framework developed earlier to analyze extinction behavior in two important classes of reaction networks that are not weakly reversible. Our first result concerns deficiency-zero networks. By constructing appropriate linear Lyapunov functions, we prove that weak extinction must occur for at least one species whenever the dynamics evolve inside a bounded stoichiometric compatibility class.

\begin{theorem} 
\label{thm:weak_extinction}

Let $G = (V, E)$ be a deficiency-zero reaction network that is not weakly reversible. 
Then, for any choice of rate constants $\bk$ and for any initial condition $\bx_0 \in \mathbb{R}^n_{>0}$ whose stoichiometric compatibility class $\mathcal{S}_{\bx_0}$ is bounded, the associated mass-action system $(G, \bk)$ admits a species $X_i$ such that $(G, \bk)$ exhibits weak extinction in $X_i$; that is, 
\[
\liminf\limits_{t \to \infty} x_i (t) = 0
\ \text{ for some } 1 \leq i \leq n.
\]
\end{theorem}

Finally, we consider linear networks that are not weakly reversible. In this case, the reaction structure ensures that certain species decay monotonically, and the associated linear Lyapunov functions enable us to establish strong extinction.

\begin{theorem} 
\label{thm:strong_extinction}

Let $G = (V, E)$ be a reaction network that is not weakly reversible, in which every source vertex consists of a single species and every target vertex is either $\emptyset$ or a single species.
Then, for any choice of rate constants $\bk$ and for any initial condition $\bx_0 \in \mathbb{R}^n_{>0}$, the associated mass-action system $(G, \bk)$ exhibits strong extinction in each species $X_i$ not belonging to a terminal strongly connected component of $G$; that is,
\[
\lim\limits_{t \to \infty} x_i (t) = 0
\ \text{ for all such species } X_i.
\]
\end{theorem}

\bigskip

\noindent
\textbf{Structure of the paper.}
In Section~\ref{sec:reaction_networks}, we review fundamental definitions and basic properties of reaction networks.
In Section~\ref{sec:extinction}, we introduce the notions of weak extinction and strong extinction.
Sections \ref{sec:strict_lyapunov} and \ref{sec:non_wr_lyapunov} focus on Lyapunov functions for reaction networks. In Theorem \ref{thm:lasalle_unique}, we apply LaSalle’s invariance principle to show that if a mass-action system admitting a strict Lyapunov function has a stoichiometric subspace containing only finitely many equilibria, then every trajectory converges to a unique nonnegative equilibrium. 
Theorem \ref{thm:stronger_not_consistent} shows that a reaction network fails to admit a positive equilibrium for any choice of rate constants if and only if it admits a linear Lyapunov function. 
%Furthermore, Proposition~\ref{prop:def_0_non_wr} provides 
Furthermore, we provide a geometric construction of such a function for deficiency-zero, non–weakly reversible networks.
In Section~\ref{sec:extinction_two_class}, we show that every deficiency-zero, non–weakly reversible network necessarily contains a species exhibiting weak extinction. %Finally, Theorem~\ref{thm:strong_extinction} 
Finally, we prove that linear non–weakly reversible networks admit a Lyapunov function, implying that each species not belonging to a terminal strongly connected component undergoes strong extinction.
Section~\ref{sec:discussion} summarizes the main results of the paper and outlines directions for future research.

\bigskip

\noindent
\textbf{Notation.}
Let $\mathbb{R}_{\geq 0}^n$ and $\mathbb{R}_{>0}^n$ denote the set of vectors in $\mathbb{R}^n$ with non-negative entries and positive entries, respectively. For vectors $\bx = (\bx_1, \ldots, \bx_n)^{\intercal}\in \RR^n_{>0}$ and $\by = (\by_1, \ldots, \by_n)^{\intercal} \in \RR^n$, define the following notation:
\begin{equation} \notag
\bx^{\by} = \bx_1^{y_{1}} \cdots \bx_n^{y_{n}} \in \mathbb{R}_{>0}.
\end{equation}
For any set $A \subseteq \mathbb{R}^n$, let $\cl(A)$ denote its closure in $\mathbb{R}^n$.

\section{Reaction Networks}
\label{sec:reaction_networks}

In this section, we recall several definitions from reaction network theory.

\begin{definition}[\cite{craciun2015toric, craciun2019polynomial,craciun2020endotactic}]
\label{def:reaction network}

\begin{enumerate}[label=(\alph*)]
\item A \defi{reaction network} $G = (V, E)$ can be represented as a directed graph in $\RR^n$, also called a \defi{Euclidean embedded graph (or E-graph)}, where $V \subset \mathbb{R}^n$ is a finite set of \defi{vertices}, and $E \subseteq V \times V$ is a finite set of directed \defi{edges}.
We assume that $G = (V, E)$ has no self-loops or isolated vertices, and contains at most one edge between any ordered pair of vertices. 

\item A directed edge $(\by,\by')\in E$, also called a \defi{reaction}, is denoted by $\by \to \by'$, where $\by$ and $\by'$ are referred to as the \defi{source vertex} and \defi{target vertex}, respectively.
The corresponding \defi{reaction vector} is defined as $\by' - \by \in\mathbb{R}^n$.
\end{enumerate}
\end{definition}

\begin{definition}[\cite{feinberg1979lectures,gunawardena2003chemical,yu2018mathematical}]
\label{def:graphic_property}

Let $G = (V, E)$ be a reaction network.
\begin{enumerate}[label=(\alph*)]
\item The set of vertices $V$ is partitioned into maximal connected components of $G$, referred to as \defi{linkage classes}.

\item A connected component of $G$ is said to be \defi{strongly connected} if every reaction in the component lies on a directed cycle. 
In particular, a linkage class of $G$ is said to be strongly connected if it satisfies this condition.
The reaction network $G$ is called \defi{weakly reversible} if each of its linkage classes is strongly connected.

\item A strongly connected component $SC$ of $G$ is said to be \defi{terminal strongly connected} if for every reaction $\by \to \by' \in E$ with $\by \in SC$, we have $\by' \in SC$. 
In other words, there are no reactions with the source vertex in $SC$ and the target vertex in $V \setminus SC$.

\item $G$ is said to be \defi{consistent} if there exists a collection of positive scalars $\big\{ \lambda_{\by \to \by'}>0 \mid \by \to \by' \in E \big\}$ such that
\begin{equation} \notag
\sum\limits_{\by\rightarrow \by'\in E}\lambda_{\by\rightarrow \by'} (\by' -\by) = 0.
\end{equation}
\end{enumerate}
\end{definition}

\begin{figure}[ht!]
\centering
\includegraphics[scale=0.38]{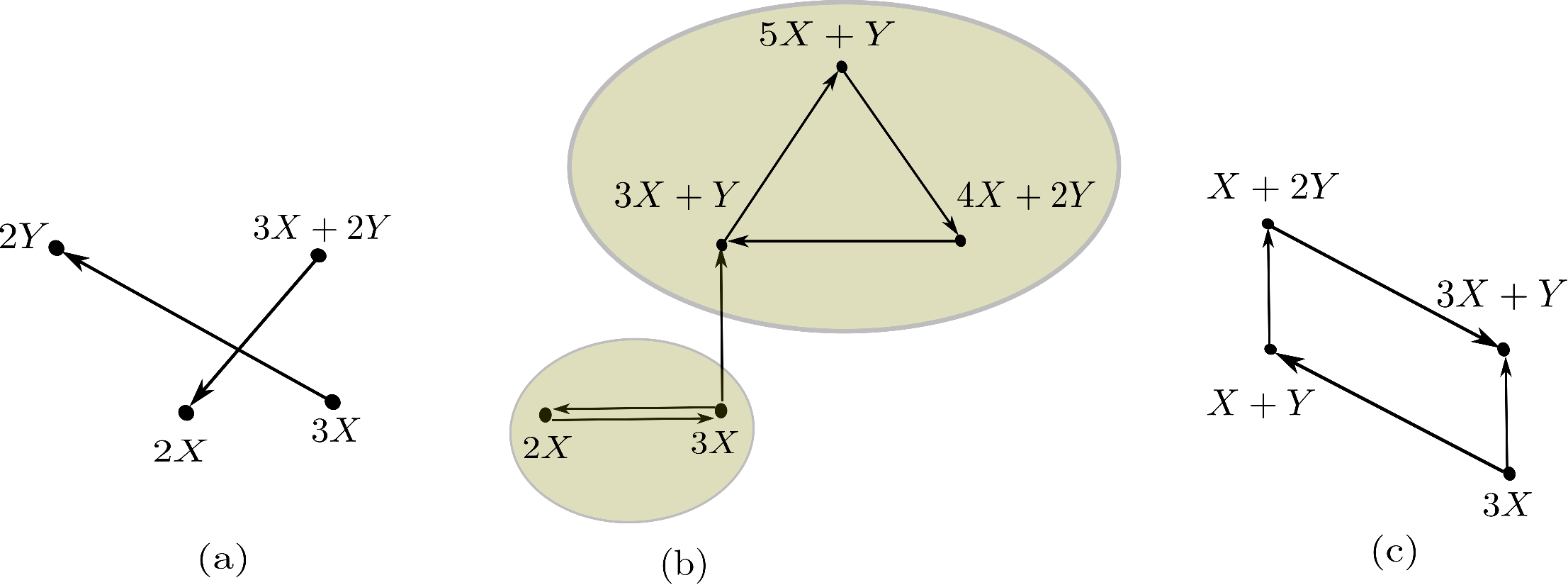}
\caption{ 
Examples of reaction networks. (a) a reaction network with two linkage classes. (b) a reaction network with one linkage class and two terminal strongly connected components (circled in light green). (c) A weakly reversible reaction network with a single, strongly connected linkage class.
}
\label{fig:euclidean_embedded_graph}
\end{figure} 

\begin{definition}[\cite{feinberg1979lectures,gunawardena2003chemical,voit2015150,yu2018mathematical}]
\label{def:mass_action}

Let $G=(V, E)$ be a reaction network. 
\begin{enumerate}[label=(\alph*)]
\item For each reaction $\by \to \by'\in E$, let $k_{\by\to \by'}$ denote the corresponding \defi{reaction rate constant}.
Define the \defi{reaction rate vector} as
\[
\bk := (k_{\by\to \by'})_{\by\to \by' \in E} \in \mathbb{R}_{>0}^{|E|}.
\]
The \defi{associated mass-action system} generated by $(G, \bk)$ is the dynamical system on  $\RR_{>0}^n$ given by
\begin{equation}
\label{eq:mas_ds}
\frac{d\bx}{dt} = \sum\limits_{\by \to \by' \in E}k_{\by\rightarrow\by'}{\bx}^{\by}(\by'-\by).
\end{equation}

\item The \defi{stoichiometric subspace} $\mathcal{S}_G$ of the reaction network $G$ is defined as
\begin{equation} \notag
\mathcal{S}_G = \spn \{ \by' - \by: \by \rightarrow \by' \in E \}.
\end{equation}
\cite{sontag2001structure} establishes that if $V \subset \mathbb{Z}_{\geq 0}^n$, then the positive orthant $\mathbb{R}_{>0}^n$ is forward-invariant\footnote{ 
A set $T \subset \R^n$ is said to be \emph{forward-invariant} with respect to a dynamical system if, for every trajectory $x(t)$ with $x(0) \in T$, we have $x(\tau) \in T$ for all $\tau \geq 0$. Furthermore, $T$ is said to be \emph{invariant} with respect to a dynamical system if $x(\tau) \in T$ for all $\tau \in \R$.
\label{fn1}} under mass-action kinetics given by \eqref{eq:mas_ds}. Moreover, for any initial condition ${\boldsymbol{x}_0}\in \mathbb{R}_{>0}^n$,  the solution to \eqref{eq:mas_ds} remains confined to the invariant set
\[
\mathcal{S}_{\bx_0} := (\bx_0 + \mathcal{S}_G) \cap \mathbb{R}_{>0}^n,
\]
which is referred to as the \defi{stoichiometric compatibility class} of $\bx_0$.
\end{enumerate}
\end{definition}

\begin{example}
\label{ex:stoichiometric_compatibility_class}

\begin{figure}[!ht]
\centering
\includegraphics[scale=0.5]{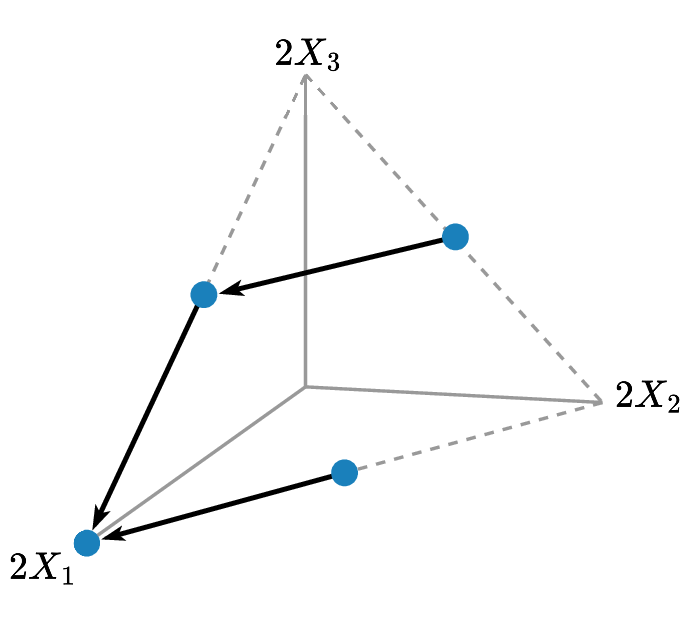}
\caption{The reaction network considered in Example~\ref{ex:stoichiometric_compatibility_class}.
}
\label{fig:LaSalles_1}
\end{figure}

Consider the reaction network $G$ shown in Figure~\ref{fig:LaSalles_1}, given by
\begin{equation} \notag
\begin{split}
X_1 + X_2 \xrightarrow[]{} 2X_1 \qquad
X_2 + X_3 \xrightarrow[]{} X_1 + X_3 \qquad
X_1 + X_3 \xrightarrow[]{} 2X_1 
\end{split}
\end{equation}
By direct computation, the stoichiometric subspace of $G$ is
\begin{equation} \notag
\mathcal{S}_G = \spn \big\{ (1, -1, 0)^{\intercal},  (1, 0, -1)^{\intercal} \big\}.
\end{equation}
For any initial condition ${\boldsymbol{x}_0} = (x_{1, 0}, x_{2, 0}, x_{3, 0}) \in \mathbb{R}_{>0}^n$, the associated stoichiometric compatibility class is
\[
\mathcal{S}_{\bx_0} := \{ \bx \in \R^3_{>0} \mid x_1 + x_2 + x_3 = x_{1, 0} + x_{2, 0} + x_{3, 0} \}.
\]
\qed
\end{example}

%\textcolor{blue}{We need a figure that illustrates ``stoichiometric compatibility class" -- maybe refer to Figure 3(a) for an illustration?}

\begin{definition}[\cite{feinberg1979lectures,gunawardena2003chemical}]

Let $(G, \bk)$ be a mass-action system. 
\begin{enumerate}[label=(\alph*)]
\item A point $\bx^* \in \mathbb{R}^n_{>0}$ is called a positive \defi{steady state} of the system if 
\begin{equation} \notag 
\displaystyle\sum_{\by\rightarrow \by' \in E } k_{\by\rightarrow\by'}{(\bx^*)}^{\by}(\by'-\by)=0.
\end{equation}

\item A point $\bx^* \in \mathbb{R}^n_{>0}$ is called a \defi{complex-balanced steady state} of the system if for each vertex $\by_0\in V$,
\begin{equation} \notag
\sum_{\by_0 \rightarrow \by \in E} k_{\by_0 \rightarrow \by} {(\bx^*)}^{\by_0} = \sum_{\by' \rightarrow \by_0 \in E} k_{\by' \rightarrow \by_0} {(\bx^*)}^{\by'}.
\end{equation}
If the mass-action system $(G, \bk)$ admits a complex-balanced steady state, then it is called a \defi{complex-balanced system}. Properties of rate constants that generate the same differential equations as complex-balanced systems have been explored in~\cite{disg_1,disg_2,disg_3,disg_4}. 
\end{enumerate}
\end{definition}

\begin{definition}[\cite{feinberg1979lectures,gunawardena2003chemical}]
\label{def:deficiency}

Let $(G, \bk)$ be a mass-action system with $\ell$ linkage classes, and let $s$ denote the dimension of the stoichiometric subspace. The \defi{(stoichiometric) deficiency} of $(G, \bk)$ is defined as 
\[
\delta= |V| - \ell -s.
\]
\end{definition}

\begin{lemma}[\cite{craciun2019realizations,feinberg2019foundations}]
\label{lem:deficiency_zero} 

Let $G=(V, E)$ be a reaction network with $\ell$ linkage classes, $V = V_1 \cup \ \cdots \ \cup \ V_{\ell}$. 
Then $\GG$ has deficiency zero if and only if the following two conditions hold:
\begin{enumerate}[label=(\alph*)]
\item The vertices within each linkage class are affinely independent; that is, for each $1 \leq i \leq \ell$, if $V_i = \{ \by_1, \ldots, \by_m \}$, then the set $\{ \by_2 - \by_1, \ldots, \by_m - \by_1 \}$ is linearly independent.

\item The stoichiometric subspaces associated with the linkage classes are linearly independent; that is, the sets
\[
\spn \big\{ \by' - \by: \by \rightarrow \by' \in E
\ \text{ with } \
\by, \by' \in V_i \big\},
\ \text{ for $i = 1, \ldots, \ell$},
\]
are linearly independent.
\end{enumerate}
\end{lemma}

\begin{definition}[\cite{feinberg1987chemical,feinberg2019foundations,khalil2002nonlinear}]

Let $\dot{\bx} = f(\bx)$ be a dynamical system, where $\bx \in D \subseteq \mathbb{R}^n_{\geq 0}$ and $f: \mathbb{R}^n \to \mathbb{R}^n$ is a continuous function.
Suppose the initial condition $\bx_0 \in D$. 
The \defi{omega-limit set} of $\bx_0$, denoted by $\omega(\bx_0)$, is defined as
\begin{equation}
\begin{split} \notag
\omega(\bx_0) = \big\{ \bx\in\R^n_{\geq 0} \mid 
& \text{ there exists an increasing sequence of times } \{ t_k \} \to \infty, 
\\& \text{ such that } \lim_{k\to\infty}\bx(t_k) = \bx
\big\}.
\end{split}
\end{equation}
In other words, $\omega(\bx_0)$ consists of all possible limit points of the trajectory starting at $\bx_0$ as time goes to infinity along some subsequence.
\end{definition}

\section{Extinction} 
\label{sec:extinction}

In this section, we focus on the concept of extinction, in which the concentrations of certain species converge to zero. We introduce two notions of extinction in reaction networks: \emph{weak extinction} and \emph{strong extinction}.

\begin{definition}
\label{defn:strong_extinction}
Let $(G, \bk)$ be a mass-action system. 

\begin{enumerate}[label=(\alph*)]
\item $(G, \bk)$ is said to exhibit \defi{weak extinction} if for some initial condition $\bx_0 \in \mathbb{R}^n_{>0}$, there exists a species $X_i$, such that the solution $\bx(t)$ satisfies
\[
\liminf\limits_{t \to \infty} x_i (t) = 0.
\]
In this case, we will say that the mass-action system $(G, \bk)$ exhibits weak extinction in species $X_i$.

\item $(G, \bk)$ is said to exhibit \defi{strong extinction} if for some initial condition $\bx_0 \in \mathbb{R}^n_{>0}$, there exists a species $X_i$ such that the solution $\bx(t)$ satisfies 
\[
\lim\limits_{t \to \infty} x_i (t) = 0.
\]
In this case, we will say that the mass-action system $(G, \bk)$ exhibits strong extinction in species $X_i$.
\end{enumerate}
% By definition, strong extinction in species $X_i$ implies weak extinction in species $X_i$.
\end{definition}

%\textcolor{blue}{The definitions above are confusing. We \underline{do not} say \underline{clearly} if we mean ``for all $x(0)$" or ``for some $x(0)$". I don't see why we would mean ``for all $x(0)$" -- this would be too strong.  We need to fix these definitions, and move them to the Introduction. (If we even need ``for all $x(0)$", then we could just say ``The system XYZW exhibits weak extinction for all initial conditions";  or ``The system ABCD exhibits strong extinction for all initial conditions".)}

In what follows, we show that certain families of reaction networks exhibit weak or strong extinction. Our analysis relies on the properties of Lyapunov functions together with LaSalle’s invariance principle.

\subsection{Lyapunov Functions for Reaction Networks}
\label{sec:strict_lyapunov}

%\textcolor{blue}{The parenthesis ``(strict) Lyapunov" is awkward -- we need to get rid of it everywhere -- and, where needed, we just list *both* Lyapunov and strict Lyapunov.}

In this section, we introduce the notion of a \emph{Lyapunov function}.
Using LaSalle’s invariance principle, we show that, under certain additional assumptions, systems with strict Lyapunov functions possess robust dynamical properties.

\begin{definition}[\cite{khalil2002nonlinear}]
\label{def:lyapunov_function} 

Let $\dot{\bx} = f(\bx)$ be a dynamical system, where $\bx \in D \subseteq \mathbb{R}^n_{\geq 0}$ and $f: \mathbb{R}^n \to \mathbb{R}^n$ is a continuous function.

\begin{enumerate}[label=(\alph*)]
\item A continuously differentiable non-constant function $\V \colon D \to \mathbb{R}$ is a \defi{Lyapunov function} for the system $\dot{\bx} = f(\bx)$ on $D$ if it satisfies
\[
\dot{\V} (\bx) = \nabla{\V}(\bx)\cdot f(\bx) \leq 0
\ \text{ for all $\bx\in D$}.
\]

\item Moreover, let $E = \big\{ \bx \in D \mid f(\bx) = 0 \big\}$. A Lyapunov function $\V\colon D \to \mathbb{R}$ is called a \defi{strict Lyapunov function} on $D$ if it further satisfies that $\dot{\V} (\bx) = 0$ if and only if $\bx \in E$.
\end{enumerate}
\end{definition}

\begin{remark}

\begin{enumerate}[label=(\alph*)]
\item Lyapunov functions are fundamental tools for analyzing the qualitative behavior of dynamical systems. In particular, the existence of a strict Lyapunov function ensures that the system admits no periodic trajectories. Moreover, if a Lyapunov function is locally positive definite, it can be used to establish the stability of an equilibrium. 

\item In reaction network theory, the Horn–Jackson Lyapunov function
\begin{equation} \label{eq:HJ_Lyapunov_function}
\sum\limits^{n}_{i = 1} x_i \ln (x_i) - x_i - x_i \ln (x^*_i),
\end{equation}
where $\bx^* = (x^*_1, \ldots, x^*_n) \in \mathbb{R}^n_{>0}$ denotes a positive complex-balanced equilibrium, characterizes the stability of complex-balanced systems. This strictly convex Lyapunov function ensures that every complex-balanced equilibrium is locally asymptotically stable~\cite{horn1972general}.
In fact, complex-balanced equilibria are conjectured to be {\em globally} asymptotically stable, a statement known as the \emph{Global Attractor Conjecture}~\cite{CraciunDickensteinShiuSturmfels2009, horn1974}. Craciun has proposed a proof of this conjecture in full generality~\cite{craciun2015toric} using the idea of toric differential inclusions~\cite{craciun2019polynomial,craciun2019quasi,ding2021minimal,ding2022minimal}.

\item There is a direct connection between extinction and the Global Attractor Conjecture. Specifically, it would suffice to prove the conjecture by showing that complex-balanced systems cannot exhibit weak extinction for any positive initial condition~\cite{craciun2013persistence}. Using LaSalle's invariance principle, the absence of extinction prevents trajectories from approaching the boundary and thereby ensures convergence to the unique complex-balanced equilibrium within each stoichiometric compatibility class.
\end{enumerate} 
\end{remark}

%\textcolor{blue}{We need to mention here the connection between extinction and the global attractor conjecture: in order to prove the global attractor conjecture it would be sufficient to prove that complex-balanced systems {\em cannot} give rise to extinction for any positive initial conditions.}

We next recall \emph{LaSalle's invariance principle}, a powerful tool that characterizes the asymptotic behavior of trajectories. 
%when used in conjunction with a Lyapunov function.
Note that LaSalle's invariance principle appears under multiple formulations. In this work, we adopt the version stated below, which is well-suited for analyzing the asymptotic behavior of trajectories in mass-action systems.

%\textcolor{blue}{Say somewhere that LaSalle's invariance principle has *multiple* formulations, and here we use the one below.}

\begin{theorem}[LaSalle's invariance principle \cite{khalil2002nonlinear}]
\label{thm:LaSalles_invariance}

Suppose a dynamical system
\begin{equation} \label{eq:lasalle}
\frac{d \bx}{d t} = f(\bx)
\ \text{ with } \
\bx\in\R^n_{\geq 0}
\ \text{ and } \
f: \R^n_{\ge 0} \to \R^n.
\end{equation}
Let $\Theta\subset\R^n_{\geq 0}$ be a compact forward-invariant set with respect to \eqref{eq:lasalle}, and let $\V : \R^n \to \R$ be a continuously differentiable function such that
\[
\dot{\V} (\bx) \leq 0
\ \text{ for all $\bx \in \Theta$.} 
\]
Define the set $\Phi = \big\{ \bx \in \Theta \mid \dot{\V} (\bx) = 0 \big\}$ and let $\varphi$ denote the largest invariant set\footref{fn1} in $\Phi$.
Then, for every solution $\bx(t)$ with initial condition $\bx_0 \in \Theta$, the omega-limit set of $\bx_0$ satisfies 
\[
\omega(\bx_0)\subseteq\varphi.
\]
\end{theorem}

We now present two results that apply LaSalle’s invariance principle to characterize the long-term behavior of dynamical systems admitting a (strict) Lyapunov function.

\begin{theorem}
\label{thm:lasalle_unique}

Let $(G, \bk)$ be a mass-action system, and let $D \subset \R^n_{\geq 0}$ be a compact forward-invariant set of $(G, \bk)$.
Suppose that $(G, \bk)$ admits a strict Lyapunov function on $D$.
If there exists a stoichiometric compatibility class $\mathcal{S}_{\hat{\bx}}$ such that $(G, \bk)$ admits only finitely many non-negative equilibria in $\cl (\mathcal{S}_{\hat{\bx}})$, then every trajectory $\bx (t)$ of $(G, \bk)$ with initial condition $\bx_0 \in \mathcal{S}_{\hat{\bx}} \cap D$ satisfies
\begin{equation} \notag
\lim\limits_{t\to\infty} \bx (t) = \bx^*
\ \text{ for some non-negative equilibrium $\bx^* \in \cl (\mathcal{S}_{\hat{\bx}})$}.
\end{equation} 
\end{theorem}

\begin{proof}

For the system $(G, \bk)$, there exists a strict Lyapunov function $\V (\bx)$ on a compact forward-invariant set $D$ such that 
\[
\dot{\V} (\bx) = \nabla{\V}(\bx)\cdot f(\bx) \leq 0
\ \text{ for all $\bx\in D$}.
\]
Moreover, $\dot{\V} (\bx) = 0$ if and only if $\bx \in \big\{ \bx \in D \mid f(\bx) = 0 \big\}$.
Since $\bx_0 \in \mathcal{S}_{\hat{\bx}} \cap D$, LaSalle’s invariance principle implies that
\begin{equation} \label{eq1:lasalle_unique}
\omega(\bx_0) \subseteq \big\{ \bx \in D \mid f(\bx) = 0 \big\}.
\end{equation}
Recall that $D$ is a compact forward-invariant set of $(G, \bk)$.
Note that omega-limit sets of bounded trajectories are connected and $(G, \bk)$ admits only finitely many non-negative equilibria in $\cl (\mathcal{S}_{\hat{\bx}})$. 
Thus, we conclude that $\omega(\bx_0)$ reduces to a single point $\bx^*$, that is,
\[
\omega(\bx_0) = \{ \bx^* \}.
\]
Together with \eqref{eq1:lasalle_unique}, it follows that $\bx^* \in \cl (\mathcal{S}_{\hat{\bx}})$ is a non-negative equilibrium of $(G, \bk)$, and
\[
\lim\limits_{t\to\infty} \bx (t) = \bx^*.
\]
\end{proof}

\begin{remark}
Theorem~\ref{thm:lasalle_unique} also applies to complex-balanced systems~\cite{horn1972general}. In particular, consider a level set of the Horn–Jackson Lyapunov function \eqref{eq:HJ_Lyapunov_function} that lies entirely within the positive orthant, i.e., does not intersect its boundary.
Let $D$ be the closed set given by the intersection of the region enclosed by this level set with a stoichiometric compatibility class. Then by Theorem~\ref{thm:lasalle_unique}, any trajectory of this complex-balanced system with initial condition in $D$ converges to a non-negative equilibrium.
\end{remark}

\begin{proposition}

Let $(G, \bk)$ be a mass-action system. Suppose that
\begin{enumerate}[label=(\alph*)]
\item every stoichiometric compatibility class is bounded, and
%\textcolor{blue}{This should be ``all stoichiometric compatibility classes are bounded"} and

\item for each stoichiometric compatibility class $\mathcal{S}_{\hat{\bx}}$, there exists a Lyapunov function $\V(\bx)$ satisfying $\dot{\V}(\bx (t)) < 0$ for all $\bx \in \mathcal{S}_{\hat{\bx}}$.
\end{enumerate}
Then every trajectory of $(G,\bk)$ exhibits weak extinction. 
%\textcolor{blue}{The wording "every trajectory of $(G,\bk)$ exhibits {\em weak extinction}.}
\end{proposition}

\begin{proof}

Consider an arbitrary stoichiometric compatibility class $\mathcal{S}_{\hat{\bx}}$.
Let $D := \cl (\mathcal{S}_{\hat{\bx}})$ denote the closure of $\mathcal{S}_{\hat{\bx}}$. 
From both assumptions, $D$ is a compact forward-invariant set of $(G, \bk)$ and there exists a Lyapunov function $\V(\bx)$ such that
\[
\dot{\V} (\bx) = \nabla{\V}(\bx)\cdot f(\bx) \leq 0
\ \text{ for all $\bx\in D$}.
\]
From LaSalle’s invariance principle,  every solution $\bx(t)$ with initial condition $\bx_0 \in \mathcal{S}_{\hat{\bx}}$ satisfies that
\[
\emptyset \neq \omega(\bx_0) \subseteq \big\{ \bx \in \Theta \mid \dot{\V} (\bx) = 0 \big\}.
\]
Recall from assumption $(b)$ that $\dot{\V} (\bx) < 0$ for all $\bx \in \mathcal{S}_{\hat{\bx}}$.
It follows that the trajectory $\bx(t)$ converges to the boundary of $\mathcal{S}_{\hat{\bx}}$, that is, $\omega(\bx_0) \subseteq \partial \R^n_{\geq 0}$. 
\end{proof}

\begin{example}
\label{ex:LaSalles_1}

Recall the reaction network $G$ shown in Figure~\ref{fig:LaSalles_1}.  
We now incorporate the reaction rate constants, which yields the following:
\begin{equation} \notag
\begin{split}
X_1 + X_2 \xrightarrow[]{k_1} 2X_1 \qquad
X_2 + X_3 \xrightarrow[]{k_2} X_1 + X_3 \qquad
X_1 + X_3 \xrightarrow[]{k_3} 2X_1 
\end{split}
\end{equation}
The corresponding mass-action system $(G, \bk)$ is given by
\begin{equation} \label{eq1:LaSalles_1}
\begin{split}
\frac{d x_1}{d t} & = k_1 x_1 x_2 + k_2 x_2 x_3 + k_3 x_1 x_3, \\
\frac{d x_2}{d t} & = - k_1 x_1 x_2 - k_2 x_2 x_3, \\
\frac{d x_3}{d t} & = - k_3 x_1 x_3. 
\end{split}
\end{equation}
Let $\bx = (x_1, x_2, x_3)$. Consider the function $\V (\bx) = - x_1 + \frac{1}{2}(x_2 + x_3)$.
By direct computation, we obtain
\begin{equation} \label{eq2:ex4}
\begin{split}
\dot{\V} (\bx) 
& = (-1 , \frac{1}{2}, \frac{1}{2})\cdot (k_1 x_1 x_2 + k_2 x_2 x_3 + k_3 x_1 x_3, - k_1 x_1 x_2 - k_2 x_2 x_3, - k_3 x_1 x_3 ) 
\\& = -\frac{3}{2}( k_1 x_1 x_2 + k_2 x_2 x_3)- \frac{1}{2} k_3 x_1 x_3.
\end{split}
\end{equation}
This implies that $\dot{\V} (\bx) \leq 0$ for all $x \in \R^3_{\geq 0}$ and $\dot{\V} (\bx) = 0$ if and only if $f(\bx) = 0$.

Following Example~\ref{ex:stoichiometric_compatibility_class}, a stoichiometric compatibility class of \eqref{eq1:LaSalles_1} is given by
\[
\mS_{\alpha} := \{ \bx \in \R^3_{>0} \mid x_1 + x_2 + x_3 = \alpha \}
\ \text{ with $\alpha > 0$}.
\]
Then $\cl(\mS_{\alpha})$ is a compact and forward-invariant set of $(G, \bk)$.
Let $\bx(t)$ be a trajectory of $(G, \bk)$ with initial condition $\bx_0 \in \mS_{\alpha}$.
From \eqref{eq2:ex4}, $\V(\bx)$ is a strict Lyapunov function on $\cl(\mS_{\alpha})$ of $(G, \bk)$.
Using Theorem~\ref{thm:lasalle_unique}, we have
\begin{equation} \notag
\lim\limits_{t\to\infty} \bx (t) = \bx^*
\ \text{ for some non-negative equilibrium $\bx^* = (x^*_1, x^*_2, x^*_3) \in \cl (\mathcal{S}_{\alpha})$}.
\end{equation} 
By direct computation, there are three non-negative equilibria in $\cl(\mS_{\alpha})$, namely $(\alpha, 0, 0)$, $(0, \alpha, 0)$, and $(0, 0, \alpha)$.
Moreover, from \eqref{eq1:LaSalles_1} we have
\[
\frac{d x_2}{d t} < 0 
\ \text{ and } \
\frac{d x_3}{d t} < 0
\ \text{ for all $\bx \in \mS_{\alpha}$}.
\]
This further implies that $x^*_2 < \alpha, x^*_3 < \alpha$, and hence $\bx^* = (\alpha, 0, 0)$. 
Therefore, for any trajectory $\bx(t)$ of $(G, \bk)$ with initial condition $\bx_0 \in \mS_{\alpha}$,
\begin{equation} \notag
\lim\limits_{t\to\infty} \bx (t) = (\alpha, 0, 0).
\end{equation}
\qed
\end{example}

The following example illustrates the necessity of all hypotheses in Theorem~\ref{thm:lasalle_unique}. In particular, it shows that if the strictness condition on the Lyapunov function is removed, trajectories may fail to converge to an equilibrium.

\begin{example}[Modified Ivanova network]
\label{ex:modified_ivanova}

\begin{figure}[!ht]
\centering
\subfloat[]{
    % width=0.7\textwidth
    \includegraphics[scale=.5]{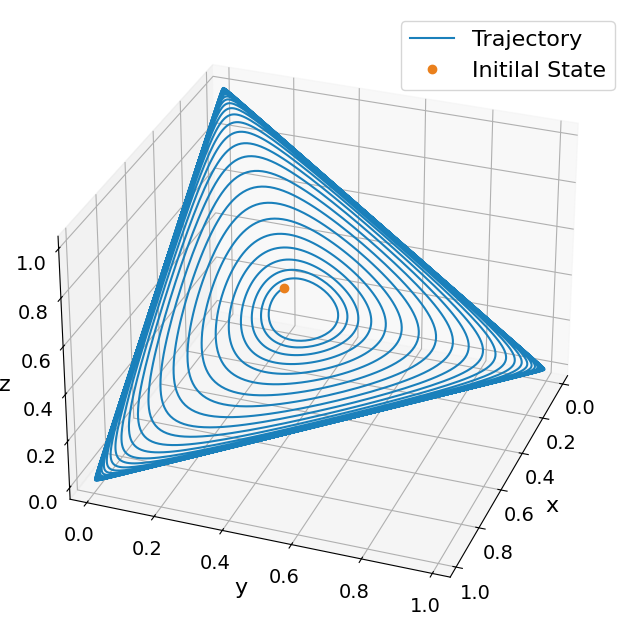}
    \label{fig:modified_ivanova_traj}
}
\vspace{1em}
\subfloat[]{
    \includegraphics[scale=.5]{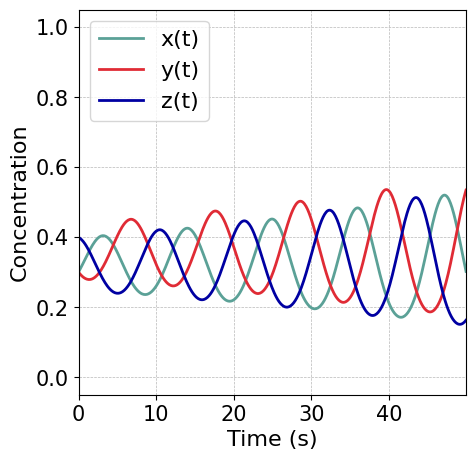}
    \label{fig:modified_ivanova_series_1}
}
\subfloat[]{
    \includegraphics[scale=.5]{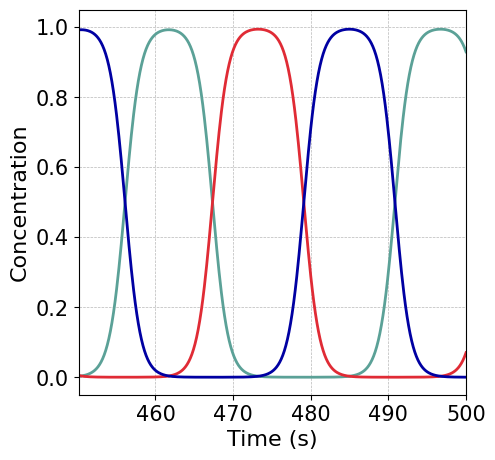}
    \label{fig:modified_ivanova_series_2}
}
\hfill
\caption{
(a) Phase space portrait of the modified Ivanova network. 
(b) Concentration of species $X$, $Y$, and $Z$ as a function of time (for the first few instances of time). (c) Concentration of species $X$, $Y$, and $Z$ as a function of time (for later instances of time).
}
\label{fig:modified_ivanova}
\end{figure}

The original Ivanova network consists of the following reactions:
\begin{equation} \notag
X + Y \rightarrow 2Y \qquad
Y + Z \rightarrow 2Z \qquad
X + Z \rightarrow 2X
\end{equation}
This network is known to exhibit periodic trajectories. 
We modify the Ivanova network by adding the reaction $2X + Y + Z \rightarrow 3X + Y$ and set all reaction rate constants to $1$, which yields the extended network:
\begin{equation} \label{eq:modified_ivanova}
X + Y \xrightarrow[]{1} 2Y \qquad
Y + Z \xrightarrow[]{1} 2Z \qquad
X + Z \xrightarrow[]{1} 2X \qquad
2X + Y + Z \xrightarrow[]{1} 3X + Y
\end{equation}
The corresponding mass-action system $(G, \bk)$ is given by
\begin{equation} \label{eq1:modified_ivanova}
\begin{split}
\frac{d x}{d t} & =  -xy + xz + x^2 yz, \\
\frac{d y}{d t} & =  xy  -yz, \\
\frac{d z}{d t} & = yz  - xz -x^2yz
\end{split}
\end{equation}

From \eqref{eq1:modified_ivanova}, we have $\frac{d }{d t} (x + y + z) \equiv 0$.
Let $\bx_0 = (x (0), y (0), z (0)) = (0.4, 0.3, 0.3)$ and consider the associated stoichiometric compatibility class
\[
\mS_{\bx_0} := \{ \bx \in \R^3_{>0} \mid x + y + z = 1 \},
\]
whose closure $\cl (\mS_{\bx_0})$ is a compact forward-invariant set for the system \eqref{eq1:modified_ivanova}.
By direct computation, there are four nonnegative equilibria in $\cl(\mS_{\bx_0})$, given by
\[
(1,0,0), \ \
(0,1,0), \ \
(0,0,1), \ \ 
(0.32, 0.34, 0.32).
\]
Numerical simulation (see Figure~\ref{fig:modified_ivanova_traj}) shows that the trajectory with initial condition $\bx_0$ evolves along outward spiraling paths. 
In other words, this trajectory exhibits weak extinction (but not strong extinction).
In particular, it satisfies that $\omega(\bx_0) = \partial \mS_{\bx_0}$, i.e., the omega-limit set is the entire boundary of the triangular stoichiometric compatibility class. 
This spiraling behavior implies that any Lyapunov function $\V(\bx)$ (if one exists) must be constant along the boundary, i.e., $\dot{\V} (\bx) = 0$ for all $\bx \in \partial \mS_{\bx_0}$.
Since points on this boundary are not all equilibria, any such Lyapunov function $\V(\bx)$ cannot be strict.

This shows that a key hypothesis of Theorem~\ref{thm:lasalle_unique} (namely, that the Lyapunov function be strict on the forward-invariant set) is not satisfied. Consequently, although the system has finitely many equilibria in $\cl(\mS_{\bx_0})$, the trajectory does not converge to any of these equilibria; instead, it spirals toward the boundary of the triangle.
\qed
\end{example}

\begin{remark}

For the mass-action system $(G, \bk)$ of the modified Ivanova network in Example~\ref{ex:modified_ivanova}, under the initial condition $\bx_0 = (0.4, 0.3, 0.3)$, the trajectory spirals toward the boundary of the triangular stoichiometric compatibility class. Consequently, all species $X, Y, Z$ exhibit weak extinction, while none of them exhibit strong extinction.
\end{remark}

\subsection{Lyapunov Functions for Deficiency-Zero, Non-Weakly Reversible Networks}
% \subsection{Lyapunov Functions for Deficiency-zero,  Networks that are not weakly reversible}
\label{sec:non_wr_lyapunov}

In this section, we analyze deficiency-zero networks that are not weakly reversible and present an explicit construction of Lyapunov functions for the associated systems.

We begin by recalling Stiemke’s Theorem, which will be used later in the subsequent analysis.

\begin{theorem}[Stiemke's Theorem {\cite{Rockafellar1970}}]
\label{thm:stiemke}

Let $\bu_1,\bu_2,...,\bu_k \in \mathbb{R}^n$. 
Then exactly one of the following statements holds:
\begin{enumerate}[label=(\alph*)]
\item There exists constants $c_1,c_2,....,c_k\in\mathbb{R}_{>0}$ such that 
\[
\sum\limits_{i=1}^k c_i\bu_i = \mathbf{0}.
\]

\item There exists a vector $\bw \in \mathbb{R}^n$ such that
\[
\bw \cdot \bu_i \leq 0
\ \text{ for all } \ 
1 \leq i \leq k,
\]
with strict inequality holding for at least one $i$.
\end{enumerate} 
\end{theorem}

\begin{theorem}
\label{thm:stronger_not_consistent}
Let $G = (V, E)$ be a reaction network. Then the following statements are equivalent:
\begin{enumerate}[label=(\alph*)]
\item $G$ is not consistent.

\item There exists a vector $\bw \in \mathbb{R}^n$ such that $\bw \cdot (\by' - \by) \leq 0$ for every reaction $\by \rightarrow \by' \in E$, and the inequality is strict for at least one reaction in $E$.

\item For all choices of rate constants $\bk$, there exists a linear Lyapunov function $\V$ for the mass-action system $(G, \bk)$ such that $\dot{\V}(\bx(t)) < 0$ for every $\bx \in \mathbb{R}^n_{>0}$.

\item For all choices of rate constants $\bk$, the mass-action system $(G, \bk)$ does not admit a positive equilibrium.
\end{enumerate} 
\end{theorem}

\begin{proof}

We show the equivalence of the four statements by establishing the following chain of implications:
\[
(a) \Rightarrow (b) \Rightarrow (c) \Rightarrow (d) 
%\Rightarrow (e) 
\Rightarrow (a).
\]

\smallskip

$(a) \Rightarrow (b)$: 
This follows directly from Theorem~\ref{thm:stiemke}. 

\smallskip

$(b) \Rightarrow (c)$: 
Define the linear function $\mathcal{V}(\bx) = \sum_{i=1}^n w_i x_i$. 
We claim that, for any choices of rate constants $\bk$, $\mathcal{V}$ is a Lyapunov function for the mass-action system $(G, \bk)$.
Differentiating $\mathcal{V}$ with respect to time along trajectories of $(G, \bk)$, we obtain
\[
\dot{\V} (\bx) = \frac{d \V}{d \bx} \cdot \frac{d\bx}{dt} = \bw \cdot \sum\limits_{\by \rightarrow \by'\in E} k_{\by \rightarrow \by'} \bx^{\by}(\by'- \by) = \sum\limits_{\by \rightarrow \by'\in E} k_{\by \rightarrow \by'} \bx^{\by} \big( \bw \cdot (\by' - \by) \big).
\]
From part $(b)$, $\bw \cdot (\by' - \by) \leq 0$ for every reaction $\by \rightarrow \by' \in E$, and there exists at least one reaction for which the inequality is strict.
Hence, 
\[
\dot{\mathcal{V}}(\bx) < 0
\ \text{ for all } \
\bx \in \mathbb{R}_{>0}^n,
\]
and we prove the claim.

\smallskip

$(c) \Rightarrow (d)$: 
Given rate constants $\bk$, let $\mathcal{V}$ be a Lyapunov function for the dynamical system $(G, \bk)$, such that
\[
\dot{\V} (\bx) = \frac{d \V}{d \bx} \cdot \frac{d\bx}{dt} < 0
\ \text{ for all }
\bx \in \mathbb{R}_{>0}^n.
\]
This implies that the system has no positive equilibrium, since $\frac{d\bx}{dt} \neq 0$ for any $\bx \in \mathbb{R}^n_{>0}$.

\smallskip

$(d) \Rightarrow (a)$: 
We prove the contrapositive: if $G$ is consistent, then there exists a choice of rate constants $\bk$ such that the mass-action system $(G, \bk)$ admits a positive equilibrium.

Since $G$ is consistent, there exists positive scalars $\big\{ \lambda_{\by \to \by'} > 0 \mid \by \to \by' \in E \big\}$ such that
\begin{equation} \notag
\sum\limits_{\by\rightarrow \by'\in E}\lambda_{\by\rightarrow \by'} (\by' -\by) = 0.
\end{equation}
Define the rate constants $\bk = (k_{\by\rightarrow \by'})_{\by\rightarrow \by' \in E} \in \mathbb{R}^{|E|}_{>0}$ by
\[
k_{\by\rightarrow \by'} = \lambda_{\by \to \by'} > 0.
\]
Consider the mass-action system $(G, \bk)$ evaluated at the point $\bx^* = (1, \ldots, 1)$. We have
\[
\frac{d\bx}{dt} \mid_{\bx^* = (1, \ldots, 1)} = \sum\limits_{\by\rightarrow \by'\in E} k_{\by\rightarrow \by'} (\mathbf{1})^{\by} (\by' -\by) = 0.
\]
Therefore, $\bx^* = (1, \ldots, 1)$ is a positive equilibrium of the system $(G, \bk)$.
\end{proof}

\begin{remark}
\label{rmk:lyapunov_function}

Note that in Theorem~\ref{thm:stronger_not_consistent}, the implication $(c)\Rightarrow (d)$ does not rely on the Lyapunov function being \emph{linear}. In fact, the following statements are unexpectedly equivalent.
\begin{enumerate}[label=(\alph*)]
    \item For all choices of rate constants $\bk$, there exists a  {\em linear Lyapunov function} $\V$ for the mass-action system $(G, \bk)$ such that $\dot{\V}(\bx(t)) < 0$ for every $\bx \in \mathbb{R}^n_{>0}$.
    \item For all choices of rate constants $\bk$, there exists a {\em Lyapunov function} $\V$ for the mass-action system $(G, \bk)$ such that $\dot{\V}(\bx(t)) < 0$ for every $\bx \in \mathbb{R}^n_{>0}$.
\end{enumerate} 
\end{remark}

\begin{lemma}[\cite{gunawardena2003chemical}] 
\label{lem:terminal}

Let $G = (V, E)$ be a reaction network that is not weakly reversible. For each linkage class $L$ of $G$ that is not weakly reversible, it admits a decomposition into disjoint strongly connected components
\begin{equation} \notag
L = SC_1 \sqcup SC_2 \sqcup \cdots \sqcup SC_k,
\end{equation}
where each $SC_i$ is a strongly connected component of $L$.
Moreover, one or more of these components, but not all, is a terminal strongly connected component.
\end{lemma}

From Theorem~\ref{thm:stronger_not_consistent}, every inconsistent reaction network $G$ admits a linear Lyapunov function for the mass-action system $(G, \bk)$. 
In the following, we provide a detailed construction of such a Lyapunov function for every deficiency-zero network that is not weakly reversible. 

\begin{proof}[Proof of Proposition \ref{prop:def_0_non_wr}]

From Theorem~\ref{thm:stronger_not_consistent}, it suffices to find a vector $\bw \in \mathbb{R}^n$ such that $\bw \cdot (\by' - \by) \leq 0$ for every reaction $\by \rightarrow \by' \in E$, and the inequality is strict for at least one reaction in $E$.
Given such a vector $\bw$, the proof of Theorem \ref{thm:stronger_not_consistent} shows that $\mathcal{V}(\bx) = \sum_{i=1}^n w_i x_i$ is a linear Lyapunov function for the mass-action system $(G, \bk)$.

We divide the analysis into two cases: (a) $G$ has a single linkage class, and (b) $G$ has multiple linkage classes.

\smallskip

\textbf{Case 1: $G$ has a single linkage class. }
% It can be verified that the reaction network $G$ decomposes into disjoint strongly connected components:
From Lemma \ref{lem:terminal}, $G$ can be decomposed into disjoint strongly connected components:

\begin{equation} \notag
V = SC_1 \sqcup SC_2 \sqcup \cdots \sqcup SC_k.
\end{equation}
% where each $SC_i$ is a strongly connected component of $G$.
Moreover, since $G$ is not weakly reversible, there exists at least one terminal strongly connected component, which we denote without loss of generality by $SC_1 \subsetneq V$. Specifically, for every reaction $\by \to \by' \in E$, if $\by \in SC_1$, then $\by' \in SC_1$. 
% In other words, there are no reactions with the source vertex in $SC_1$ and the target vertex in $V \setminus SC_1$.

Let $V_1$ be a connected component of $V \setminus SC_1$, and define $V_2 = V \setminus V_1$.
We now consider two reaction networks, denoted by $G_1 = (V_1, E_1)$ and $G_2 = (V_2, E_2)$.
The first reaction network $G_1$ consists of the vertex set $V_1$ and all reactions whose source and target vertices both lie in $V_1$.
The second reaction network $G_2$ consists of the vertex set $V_2$ and all reactions whose source and target vertices both lie in $V_2$.
Since $G$ has a single linkage class and $SC_1$ is a terminal strongly connected component of $G$, it follows that every connected component of $V \setminus ( SC_1 \cup V_1 )$ can connect only to $SC_1$. Hence, $G_2$ has a single linkage class and is therefore connected.

Let $\mS_{G}, \mS_{G_1}, \mS_{G_2}$ denote the stoichiometric subspaces corresponding to the reaction networks $G$, $G_1$, and $G_2$, respectively. Since $G$ has a single linkage class and $V_1$ is a terminal strongly connected component of $G$, there exists a reaction $\by_2 \to \by_1 \in E$ such that $\by_2 \in V_2$ and $\by_1 \in V_1$.
Since both $G_1$ and $G_2$ are connected, for any reaction $\by \to \by' \in E$, we can express the reaction vector as 
\[
\by' - \by' = \by' - \by_1 + (\by_1 - \by_2) + \by_2 - \by'.
\]
This implies that
\begin{equation} \label{eq2:def_0_non_wr}
\mS_{G} = \mS_{G_1} + \mS_{G_2} + \spn \{ \by_1 - \by_2 \}.
\end{equation}
Since $G$ is a deficiency-zero network with a single linkage class, Definition~\ref{def:deficiency} gives
\[
\dim (S_G) = |V| - 1.
\]
On the other hand, $G_1$ and $G_2$ satisfy that
\[
\dim (\mS_{G_1}) \leq |V_1| - 1
\ \text{ and } \
\dim (\mS_{G_2}) \leq |V_2| - 1.
\]
Combining these with \eqref{eq2:def_0_non_wr} and the identity $|V| = |V_1| + |V_2|$, we derive
\[
\dim (S_G) = \dim (\mS_{G_1}) + \dim (\mS_{G_2}) + 1.
\]
Hence, there exists a vector $\bw \in S_G$ such that
\[
\bw \perp \mS_{G_1} + \mS_{G_2}.
\]
In particular, we can choose $\bw$ so that $\bw \cdot (\by_1 - \by_2) < 0$. From \eqref{eq2:def_0_non_wr}, it follows that
\[
\bw \cdot (\by' - \by) \leq 0
\ \text{ for all } \
\by \rightarrow \by' \in E,
\]
with strict inequality for at least one reaction (specifically, $\by_2 \to \by_1$) in $E$.

\smallskip

\textbf{Case 2: $G$ has multiple linkage classes. }
Now suppose that $G$ has $\ell > 1$ linkage classes. Then $G$ can be decomposed as
\begin{equation} \label{eq3:def_0_non_wr}
V = L_1 \sqcup L_2 \sqcup \cdots \sqcup L_k,
\end{equation}
where each $L_i$ is a linkage class of $G$.
Since $G$ is not weakly reversible, there exists at least one linkage class that is not weakly reversible. Without loss of generality, we denote this linkage class by $L_1$.

We proceed as in the single linkage class case, now applied to $L_1$.
Let $SC_1 \subsetneq L_1$ be a terminal strongly connected component of $L_1$, and let $V_1$ be a connected component of $L_1 \setminus SC_1$.
Define $V_2 = V \setminus V_1$, we consider two reaction networks:
\[
G_1 = (V_1, E_1)
\ \text{ and } \
G_2 = (V_2, E_2),
\]
where $G_1$ consists of the vertex set $V_1$ and all reactions whose source and target vertices both lie in $V_1$, while $G_2$ consists of the vertex set $V_2$ and all reactions whose source and target vertices both lie in $V_2$.

From Lemma~\ref{lem:deficiency_zero}, the stoichiometric subspaces corresponding to distinct linkage classes are linearly independent.
Analogous to the single linkage class case, we deduce that
\[
\mS_{G} = \mS_{G_1} + \mS_{G_2} + \spn \{ \by_1 - \by_2 \},
\]
and moreover,
\[
\dim (S_G) = \dim (\mS_{G_1}) + \dim (\mS_{G_2}) + 1.
\]
The rest of the proof proceeds as in Case 1. We may choose a vector $\bw \in S_G$ such that $\bw \perp \mS_{G_1} + \mS_{G_2}$ and $\bw \cdot (\by_1 - \by_2) < 0$. It then follows that
\[
\bw \cdot (\by' - \by) \leq 0
\ \text{ for all } \
\by \rightarrow \by' \in E,
\]
with strict inequality for at least one reaction 
(specifically, $\by_2 \to \by_1$) 
in $E$.
\end{proof}

The following result is a direct consequence of Theorem \ref{thm:stronger_not_consistent} and Proposition~\ref{prop:def_0_non_wr}. It was first proved by a different method in \cite{horn1972general}.

\begin{theorem}
\label{thm:def_zero_non_wr}

Let $G = (V, E)$ be a deficiency-zero reaction network that is not weakly reversible. Then for any choice of rate constants $\bk$, the mass-action system $(G, \bk)$ does not admit a positive equilibrium.
\end{theorem}

\begin{remark}[\cite{boros2019existence, deshpande2022source}]
\label{rmk:wr_consistent}

On the other hand, it was shown in \cite{deshpande2022source} that if a mass-action system $(G, \bk)$ does not admit a positive equilibrium for any choice of rate constants $\bk$, then the reaction network $G$ is not weakly reversible. Consequently, every weakly reversible reaction network $G$ is consistent. 
Moreover, it has been shown that every weakly reversible mass-action system possesses a positive equilibrium within each stoichiometric compatibility class \cite{boros2019existence}.
\end{remark}

\subsection{Extinction for Two Classes of Reaction Networks}
\label{sec:extinction_two_class}

In this section,  we identify two classes of reaction networks in which at least one species goes extinct by using Lyapunov functions and LaSalle's invariance principle.

First, we show that every deficiency-zero, non–weakly reversible network admits a species exhibiting weak extinction. 

\begin{proof}[Proof of Theorem \ref{thm:weak_extinction}]

Since $G$ is a deficiency-zero reaction network that is not weakly reversible, Proposition~\ref{prop:def_0_non_wr} shows that $(G, \bk)$ admits a linear Lyapunov function $\V(\bx) = \bw \cdot \bx$ such that 
\begin{equation} \label{eq1:weak_extinction}
\frac{d}{dt} \V(\bx(t)) = \bw \cdot \frac{d\bx}{dt} 
= \sum\limits_{\by \to \by' \in E} k_{\by\rightarrow\by'}{\bx}^{\by} \bw \cdot (\by'-\by) < 0 
\ \text{ for all } \ 
\bx \in \mathbb{R}^n_{>0}.
\end{equation}
In particular, we have $\bw \cdot \bf (\by' - \by) \leq 0$ for all $\by \rightarrow \by' \in E$, where the inequality is strict for at least one reaction. 
This, together with \eqref{eq1:weak_extinction}, implies that 
\[
\frac{d}{dt} 
\V(\bx(t)) = 0
\ \text{ only if } \ \bx \in \partial \R^n_{>0}.
\]

Given an initial condition $\bx_0$ whose stoichiometric compatibility class $\mathcal{S}_{\bx_0}$ is bounded, LaSalle’s invariance principle implies that the corresponding solution $\bx(t)$ satisfies
\[
\emptyset \neq \omega(\bx_0) \subseteq \big\{ \bx \in \cl (\mathcal{S}_{\bx_0}) \cap \R^n_{\geq 0} \mid \dot{\V} (\bx) = 0 \big\}.
\]
Therefore, $\omega(x(0))\subseteq \partial \R^n_{>0}$ and there exists a species $X_i$ such that $\liminf\limits_{t \to \infty} x_i (t) = 0$.
\end{proof}

\begin{lemma}[Barbalat’s Lemma {\cite{Khalil2017}}]\label{lem:barbalat}

Let $f(t)$ be a continuously differentiable function satisfying the following conditions:
\begin{enumerate}[label=(\alph*)]
\item $\lim\limits_{t \to \infty} f(t) = c$, where $c$ is is a finite number.

\item The derivative $f' (t)$ is uniformly continuous.
\end{enumerate} 
Then $\lim\limits_{t \to \infty} f' (t) = 0$.
\end{lemma}

Next, we prove that linear non–weakly reversible networks exhibit strong extinction in species that do not belong to a terminal strongly connected component.

\begin{proof}[Proof of Theorem \ref{thm:strong_extinction}]

Since every vertex of $G$ is either $\emptyset$ or of the form $X_i$, species in distinct linkage classes do not interact, and the associated systems within each linkage class are independent. 
We can therefore assume, without loss of generality, that $G$ has a single linkage class, since if the result holds for a single linkage class, the same analysis applies to each class individually, and the conclusions extend to the entire system.

Since the reaction network $G$ is not weakly reversible, Lemma \ref{lem:terminal} implies that $G$ contains both terminal and non-terminal strongly connected components. Let $T$ denote the union of terminal strongly connected components, and let $N$ denote the union of non-terminal strongly connected components of $G$.
Hence, there exists at least one reaction with a source species in $N$ and a target species in $T$, but no reaction in the reverse direction.
Moreover, since $\emptyset$ can only appear as a target vertex, then $\emptyset \in T$ if $\emptyset \in V$

Given any $\alpha > 0$, consider the set  
\[
D := \{ \bx \in \mathbb{R}^n_{\geq 0} \mid x_1 + \cdots + x_n \leq \alpha \}.
\]
By assumption, every reaction of $G$ satisfies
\[
(1, 1, \ldots, 1) \cdot (\by' - \by) \leq 0
\ \text{ with } \
\by \rightarrow \by' \in E.
\]
It follows that $D$ is a compact forward-invariant set for $(G, \bk)$.
Define the function: 
\[
\V (\bx) := \sum\limits_{X_i \in N} \bx_i.
\]

We now claim that $\V (\bx)$ is a Lyapunov function for $(G, \bk)$ on $D$. By direct computation, we have
\begin{equation} \label{eq1:strong_extinction}
\dot{\V} (\bx)
= \frac{d}{dt} \big( \sum\limits_{X_i \in N} \bx_i \big) = \sum\limits_{X_i \in N} \frac{d\bx_i}{dt}.
\end{equation}
Note that if $X_i \to X_j \in E$ with $X_i, X_j \in N$, then the right-hand side of $\frac{d \bx_i}{dt}$ includes the term $- k_{X_i \to X_j} x_i$ while the right-hand side of $\frac{d \bx_j}{dt}$ includes the term $k_{X_i \to X_j} x_i$. Thus, these two terms cancel each other in $\dot{\V} (\bx)$. 
This, together with \eqref{eq1:strong_extinction}, implies that
\begin{equation} \label{eq2:strong_extinction}
\dot{\V} (\bx) = \sum\limits_{X_i\in N} \sum\limits_{\substack{\text{$X_j \in T$} \\ \text{$X_i\to X_j\in E$}}} - k_{X_i\to X_j} x_i \leq 0,
\end{equation}
which proves the claim.
From \eqref{eq2:strong_extinction}, we also deduce that $\dot{\V} (\bx) = 0$ if and only if $x_i = 0$ for all $X_i \in N$ with $X_i \to X_j \in E$ and $X_j \in T$ (we denote this subset of species by $N_1$).
From LaSalle's invariance principle, every solution $\bx(t)$ with initial condition $\bx_0 \in D$ satisfies 
\[
\emptyset \neq \omega(\bx_0) \subseteq \big\{ \bx \in D \mid \dot{\V} (\bx) = 0 \big\}.
\]
Therefore, for every species $X_i \in N_1$, we have $\lim\limits_{t\to\infty} x_i(t) = 0$.

Next, we claim that $(G, \bk)$ exhibits strong extinction for every species $X_i \in N \setminus N_1$ for which there exists a reaction $X_i \to X_j \in E$ with $X_j \in N_1$ (we denote this subset of species by $N_2$).
From \eqref{eq2:strong_extinction} and the fact that $D \subset \R^n_{\geq 0}$ is forward-invariant, $\V (\bx (t))$ is non-increasing with respect to $t$ and it is non-negative. Hence, we derive that
\[
\lim\limits_{t\to\infty} \V (\bx (t)) =
\lim\limits_{t\to\infty} \Big( \sum\limits_{X_i \in N} \bx_i \Big) = C \geq 0.
\]
Since $\lim\limits_{t\to\infty} x_i(t) = 0$ for every species $X_i \in N_1$, we further obtain
\begin{equation} \label{eq3:strong_extinction}
\lim\limits_{t\to\infty} \Big( \sum\limits_{X_i \in N \setminus N_1} \bx_i \Big) = C \geq 0,
\end{equation}
where $C$ is a constant.
We now consider the following function: 
\[
\V_1 (\bx) := \sum\limits_{X_i \in N \setminus N_1} \bx_i.
\]
Similar to $\V (\bx)$, we note that if $X_i \to X_j \in E$ with $X_i, X_j \in N \setminus N_1$, then the corresponding terms in $\frac{d \bx_i}{dt}$ and $\frac{d \bx_j}{dt}$ cancel each other in $\dot{\V_1}(\bx)$. 
This further implies that
\begin{equation} \label{eq4:strong_extinction}
\dot{\V_1} (\bx) = \sum\limits_{X_i\in N \setminus N_1} \sum\limits_{\substack{\text{$X_j \in N_1$} \\ \text{$X_i\to X_j\in E$}}} - k_{X_i\to X_j} x_i + \sum\limits_{X_l \in N_1} \sum\limits_{\substack{\text{$X_m \in N \setminus N_1$} \\ \text{$X_l \to X_m \in E$}}} k_{X_l \to X_m} x_l.
\end{equation}
From \eqref{eq3:strong_extinction}, the function $\V_1(\bx)$ converges to a finite limit as $t \to \infty$.
This, together with \eqref{eq4:strong_extinction}, further implies that the function $\V_1(\bx)$ is continuously differentiable and $\dot{\V}_1(\bx)$ is uniformly continuous. 
By Barbalat's Lemma~\ref{lem:barbalat}, it follows that 
\[
\lim\limits_{t\to\infty} \dot{\V}_1 (\bx) = \lim\limits_{t\to\infty} \left( \sum\limits_{X_i\in N \setminus N_1} \sum\limits_{\substack{\text{$X_j \in N_1$} \\ \text{$X_i\to X_j\in E$}}} - k_{X_i\to X_j} x_i + \sum\limits_{X_l \in N_1} \sum\limits_{\substack{\text{$X_m \in N \setminus N_1$} \\ \text{$X_l \to X_m \in E$}}} k_{X_l \to X_m} x_l \right) = 0.
\]
Since the species in $N_1$ exhibit strong extinction and $D \subset \R^n_{\geq 0}$ is forward-invariant, it follows that
$\lim\limits_{t\to\infty} x_i(t) = 0$ for all $X_i \in N \setminus N_1$ with $X_i \to X_j \in E$ and $X_j \in N_1$ (we denote this subset of species by $N_2$). This completes the proof of the claim.

Analogously, we consider the following function: 
\[
\V_2 (\bx) := \sum\limits_{X_i \in N \setminus \{ N_1 \cup N_2 \} } \bx_i.
\]
Following a similar argument, we obtain
\begin{equation} \notag
\dot{\V_2} (\bx) = \sum\limits_{X_i \in N \setminus \{ N_1 \cup N_2 \}} \sum\limits_{\substack{\text{$X_j \in N_2$} \\ \text{$X_i\to X_j\in E$}}} - k_{X_i\to X_j} x_i + \sum\limits_{X_l \in N_2} \sum\limits_{\substack{\text{$X_m \in N \setminus \{ N_1 \cup N_2 \}$} \\ \text{$X_l \to X_m \in E$}}} k_{X_l \to X_m} x_l,
\end{equation}
and we further derive that
\[
\lim\limits_{t\to\infty} x_i(t) = 0
\ \text{for all $X_i \in N \setminus \{ N_1 \cup N_2 \}$ with $X_i \to X_j \in E$ and $X_j \in N_2$}.
\]
Iterating this argument, we conclude that $(G, \bk)$ exhibits strong extinction for every species $X_i \in N$.
\end{proof}

\section{Discussion} \label{sec:discussion}

This work establishes a connection between the existence of Lyapunov functions and extinction in reaction networks. 
In contrast to the \emph{persistence conjecture}~\cite{anderson2011proof,craciun2013persistence,feinberg1987chemical,gopalkrishnan2014geometric}, the concept of extinction has received comparatively less attention.
We introduce two types of extinction: \emph{weak extinction} and \emph{strong extinction}, and relate them to the existence of Lyapunov functions via LaSalle’s invariance principle.
Our results show that trajectories of mass-action systems admitting strict Lyapunov functions converge to a unique nonnegative equilibrium. Moreover, we construct a linear Lyapunov function for deficiency-zero, non–weakly reversible networks and show that such systems exhibit weak extinction in at least one species when the stoichiometric compatibility class is bounded.
Finally, for a non–weakly reversible reaction network, we prove that strong extinction occurs in any species not belonging to a terminal strongly connected component when the network is first-order.

While beyond the scope of the present work, the framework developed here suggests several directions for future research.
For deficiency-zero systems that are not weakly reversible, it remains to determine conditions under which one can identify the species that undergo weak extinction. 
This involves understanding how features of the network structure or stoichiometric constraints influence which species approach extinction along trajectories. 
A closely related question is to characterize additional structural or dynamical assumptions under which such systems exhibit strong extinction in a specified species.

A second direction is to obtain a more complete characterization of the capacity for weak or strong extinction in reaction networks with low-dimensional stoichiometric subspaces, particularly in one or two dimensions. For example, in the two-dimensional setting, one may ask whether there exists a ``bad'' reaction that inevitably induces weak or strong extinction for some or all choices of rate constants and for some or all initial conditions.

\bibliographystyle{plain}
\bibliography{Bibliography}

@article{craciun2013persistence,
  title={Persistence and permanence of mass-action and power-law dynamical systems},
  author={Craciun, G. and Nazarov, F. and Pantea, C.},
  journal={SIAM J. Appl. Math.},
  volume={73},
  number={1},
  pages={305--329},
  year={2013},
  publisher={SIAM}
}

@article{gopalkrishnan2014geometric,
  title={A geometric approach to the global attractor conjecture},
  author={Gopalkrishnan, M. and Miller, E. and Shiu, A.},
  journal={SIAM J. Appl. Dyn. Syst.},
  volume={13},
  number={2},
  pages={758--797},
  year={2014},
  publisher={SIAM}
}

@article{horn1972general,
  title={General mass action kinetics},
  author={Horn, F. and Jackson, R.},
  journal={Arch. Ration. Mech. Anal.},
  volume={47},
  number={2},
  pages={81--116},
  year={1972},
  publisher={Springer}
}

@article{feinberg1979lectures,
  title={Lectures on chemical reaction networks},
  author={Feinberg, M.},
  journal={Notes of lectures given at the Mathematics Research Center, University of Wisconsin},
  pages={49},
  year={1979}
}

@article{voit2015150,
  title={150 years of the mass action law},
  author={Voit, E. and Martens, H. and Omholt, S.},
  journal={PLOS Comput. Biol.},
  volume={11},
  number={1},
  pages={e1004012},
  year={2015},
  publisher={Public Library of Science}
}

@article{yu2018mathematical,
  title={Mathematical {A}nalysis of {C}hemical {R}eaction {S}ystems},
  author={Yu, P. and Craciun, G.},
  journal={Isr. J. Chem.},
  volume={58},
  number={6-7},
  pages={733--741},
  year={2018},
  publisher={Wiley Online Library}
}

@article{gunawardena2003chemical,
  title={Chemical reaction network theory for in-silico biologists},
  author={Gunawardena, J.},
  journal={Notes available for download at http://vcp. med. harvard. edu/papers/crnt. pdf},
  year={2003}
}

@article{craciun2015toric,
  title={Toric differential inclusions and a proof of the global attractor conjecture},
  author={Craciun, G.},
  journal={arXiv preprint arXiv:1501.02860},
  year={2015}
}

@article{craciun2019polynomial,
  title={Polynomial dynamical systems, reaction networks, and toric differential inclusions},
  author={Craciun, G.},
  journal={SIAM J. Appl. Algebra Geom.},
  volume={3},
  number={1},
  pages={87--106},
  year={2019},
  publisher={SIAM}
}

@article{craciun2020endotactic,
  title={Endotactic networks and toric differential inclusions},
  author={Craciun, G. and Deshpande, A.},
  journal={SIAM J. Appl. Dyn. Syst.},
  volume={19},
  number={3},
  pages={1798--1822},
  year={2020},
  publisher={SIAM}
}

@article{craciun2019quasi,
title = {Quasi-toric differential inclusions},
author = {Craciun, C. and Deshpande, A. and Yeon, J.},
journal = {Discrete Continuous Dyn. Syst. Ser. B.},
volume = {26},
number = {5},
pages = {2343-2359},
year = {2021},
}

@book{khalil2002nonlinear,
  title={Nonlinear systems},
  author={Khalil, H. and Grizzle, J.},
  volume={3},
  year={2002},
  publisher={Prentice hall Upper Saddle River, NJ}
}

@book{feinberg2019foundations,
  title={Foundations of chemical reaction network theory},
  author={Feinberg, M.},
  year={2019},
  publisher={Springer}
}

@article{disg_3,
  title={The dimension of the disguised toric locus of a reaction network},
  author={Craciun, Gheorghe and Deshpande, Abhishek and Jin, Jiaxin},
  journal={Studies in Applied Mathematics},
  volume={154},
  number={6},
  pages={e70071},
  year={2025},
  publisher={Wiley Online Library}
}

@article{disg_4,
  title={The Computation of the Disguised Toric Locus of Reaction Networks},
  author={Craciun, G. and Deshpande, A. and Jin, J.},
  journal={arXiv preprint arXiv:2412.02620},
  year={2024}
}

@article{pantea2012persistence,
  title={On the persistence and global stability of mass-action systems},
  author={Pantea, C.},
  journal={SIAM J. Math. Anal.},
  volume={44},
  number={3},
  pages={1636--1673},
  year={2012},
  publisher={SIAM}
}

@article{ding2022minimal,
  title={Minimal invariant regions and minimal globally attracting regions for toric differential inclusions},
  author={Ding, Y. and Deshpande, A. and Craciun, G.},
  journal={Adv. Appl. Math.},
  volume={136},
  pages={102307},
  year={2022},
  publisher={Elsevier}
}

@article{ding2021minimal,
  title={Minimal invariant regions and minimal globally attracting regions for variable-k reaction systems},
  author={Ding, Y. and Deshpande, A. and Craciun, G.},
  journal={Discrete Contin. Dyn. Syst. - B},
  volume={28},
  number = {3},
  pages={696-1718},
  year={2023},
  publisher={Southwest Missouri State University}
}

@article{anderson2011proof,
  title={A proof of the global attractor conjecture in the single linkage class case},
  author={Anderson, D.},
  journal={SIAM J. Appl. Math.},
  volume={71},
  number={4},
  pages={1487--1508},
  year={2011},
  publisher={SIAM}
}

@article{feinberg1977chemical,
  title={Chemical mechanism structure and the coincidence of the stoichiometric and kinetic subspaces},
  author={Feinberg, M. and Horn, F.},
  journal={Arch. Ration. Mech. Anal.},
  volume={66},
  number={1},
  pages={83--97},
  year={1977},
  publisher={Springer}
}

@article{feinberg1987chemical,
  title={Chemical reaction network structure and the stability of complex isothermal reactors - {I}. The deficiency zero and deficiency one theorems},
  author={Feinberg, M.},
  journal={Chem. Eng. Sci.},
  volume={42},
  number={10},
  pages={2229--2268},
  year={1987},
  publisher={Elsevier}
}

@article{craciun2019realizations,
  title={Realizations of kinetic differential equations},
  author={Craciun, G. and Johnston, M. and Szederk{\'e}nyi, G. and Tonello, E. and T{\'o}th, J. and Yu, P.},
  journal={arXiv preprint arXiv:1907.07266},
  year={2019}
}

@article{deshpande2022source,
  title={Source-only realizations, weakly reversible deficiency one networks, and dynamical equivalence},
  author={Deshpande, A.},
  journal={SIAM J. Appl. Dyn.},
  volume={22},
  number={2},
  pages={1502--1521},
  year={2023},
  publisher={SIAM}
}

@article{boros2019existence,
  title={Existence of positive steady states for weakly reversible mass-action systems},
  author={Boros, B.},
  journal={SIAM J. Math. Anal.},
  volume={51},
  number={1},
  pages={435--449},
  year={2019},
  publisher={SIAM}
}

@article{disg_1,
  title={On the connectivity of the disguised toric locus of a reaction network},
  author={Craciun, G. and Deshpande, A. and Jin, J.},
  journal={J. Math. Chem.},
  volume={62},
  number={2},
  pages={386--405},
  year={2024},
  publisher={Springer}
}

@article{disg_2,
  title={A Lower Bound on the Dimension of the-Disguised Toric Locus of a Reaction Network},
  author={Craciun, G. and Deshpande, A. and Jin, J.},
  journal={SIAM J. Appl. Algebra Geom.},
  volume={8},
  number={4},
  pages={877--912},
  year={2024},
  publisher={SIAM}
}

@article{CraciunDickensteinShiuSturmfels2009,
	author = {Craciun, G. and Dickenstein, A. and Shiu, A. and Sturmfels, B.},
	title = {Toric Dynamical Systems},
    journal = {J. Symbolic Comput.},
	volume = {44},
	number = {11},
	pages = {1551--1565},
	year = {2009},
}

@article{sontag2001structure,
  title={Structure and stability of certain chemical networks and applications to the kinetic proofreading model of T-cell receptor signal transduction},
  author={Sontag, E.},
  journal={IEEE Trans. Automat.},
  volume={46},
  number={7},
  pages={1028--1047},
  year={2001},
  publisher={IEEE}
}

@article{may1972stable,
author={May, R.},
title={Will a Large Complex System be Stable?},
journal={Nature},
year={1972},
month={Aug},
day={01},
volume={238},
number={5364},
pages={413-414},
issn={1476-4687},
}

@article{johnston2018conditions,
  title={Conditions for extinction events in chemical reaction networks with discrete state spaces},
  author={Johnston, M. and Anderson, D. and Craciun, G. and Brijder, R.},
  journal={J. Math. Biol.},
  volume={76},
  number={6},
  pages={1535--1558},
  year={2018},
  publisher={Springer}
}

@article{johnston2017computational,
  title={A computational approach to extinction events in chemical reaction networks with discrete state spaces},
  author={Johnston, M.},
  journal={Mathematical Biosciences},
  volume={294},
  pages={130--142},
  year={2017},
  publisher={Elsevier}
}

@book{Khalil2017,
    author    = {Khalil, Hassan K.},
    title     = {Nonlinear Systems},
    edition   = {Fourth},
    publisher = {Pearson},
    year      = {2017},
    address   = {Hoboken, NJ}
}

@book{Rockafellar1970,
    author    = {Rockafellar, R. Tyrrell},
    title     = {Convex Analysis},
    publisher = {Princeton University Press},
    year      = {1970},
    address   = {Princeton, NJ},
    series    = {Princeton Mathematical Series},
    volume    = {28}
}

@article{horn1974,
  title={The dynamics of open reaction systems: mathematical aspects of chemical and biochemical problems and quantum chemistry},
  author={Horn, F.},
  journal={SIAM-AMS Proceedings},
  volume={VIII},
  pages={125--137},
  year={1974}
}

@article{dowdle1998principles,
  title={The principles of disease elimination and eradication},
  author={Dowdle, W.},
  journal={Bull. WHO},
  volume={76},
  pages={22--25},
  year={1998},
  publisher={WORLD HEALTH ORGANISATION}
}

@article{henderson2013lessons,
  title={Lessons from the eradication of smallpox: an interview with DA Henderson},
  author={Henderson, D. and Klepac, P.},
  journal={Philos. Trans. R. Soc. B, Biol. Sci.},
  volume={368},
  number={1623},
  pages={20130113},
  year={2013},
  publisher={The Royal Society}
}

@misc{bidlack2002casarett,
  title={Casarett \& {D}oull’s toxicology: the basic science of poisons},
  author={Bidlack, W.},
  year={2002},
  publisher={Taylor \& Francis}
}

@book{goodman1996goodman,
  title={Goodman and Gilman's the pharmacological basis of therapeutics},
  author={Goodman, L.},
  volume={1549},
  year={1996},
  publisher={McGraw-Hill New York}
}

@article{rowland2011clinical,
  title={Clinical pharmacokinetics and pharmacodynamics: concepts and applications},
  author={Rowland, M. and Tozer, T.},
  journal={(No Title)},
  year={2011}
}

\end{document}